\documentclass[12pt,reqno]{amsart}
\usepackage[all]{xy}
\usepackage{amssymb}
\usepackage{amsthm}
\usepackage{hyperref}
\hypersetup{colorlinks=true,linkcolor=blue,citecolor=magenta}
\usepackage{amsmath}
\usepackage{amscd,enumitem}
\usepackage{verbatim}
\usepackage{eurosym}
\usepackage{float}
\usepackage{color}
\usepackage{dcolumn}
\usepackage[mathscr]{eucal}
\usepackage[all]{xy}
\usepackage{hyperref}
\usepackage{bbm}
\usepackage[textheight=8.65in, textwidth=6.5in]{geometry}
\newtheorem*{conj*}{Conjecture}
\newtheorem{theorem}{Theorem}[section]

\newtheorem{corollary}[theorem]{Corollary}

\theoremstyle{remark}
\newtheorem*{example}{Example}
\newtheorem*{rmk}{Remark}

\newcommand{\Z}{\mathbb{Z}}
\newcommand{\Q}{\mathbb{Q}}
\newcommand{\R}{\mathbb{R}}
\newcommand{\N}{\mathbb{N}}

\newcommand{\SL}{\operatorname{SL}}
\newcommand{\C}{\mathbb{C}}
\newcommand{\re}{\text {\rm Re}}
\newcommand{\im}{{\text {\rm Im}}}
\newcommand{\Hol}{\text {\rm Hol}}

\newcommand{\Log}{\operatorname{Log}}
\newcommand{\Arg}{\operatorname{Arg}}
\numberwithin{equation}{section}

\begin{document}
\title[$q$-brackets of $t$-hook functions]{Eichler integrals of Eisenstein series as $q$-brackets of weighted $t$-hook functions on partitions}
\author{Kathrin Bringmann}
\address{Department of Mathematics and Computer Science, Division of Mathematics, University of Cologne,
	Weyertal 86-90, 50931 Cologne, Germany}
\email{kbringma@math.uni-koeln.de}

\author{Ken Ono}
\address{Department of Mathematics, University of Virginia, Charlottesville, VA 22904}
\email{ko5wk@virginia.edu}

\author{Ian Wagner}
\address{Department of Mathematics, Vanderbilt University, Nashville, TN 37212}
\email{ian.c.wagner@vanderbilt.edu}

\dedicatory{In fond memory of Dick Askey}

\thanks{The research of the first author is supported by the Alfried Krupp Prize for Young University Teachers of the Krupp foundation and has received funding from the European Research Council (ERC) under the European Union's Horizon 2020 research and innovation programme (grant agreement No. 101001179).
The second author is grateful for the support of the Thomas Jefferson Fund and the NSF (DMS-1601306).}
\keywords{$t$-hooks, partitions, $q$-brackets, Eichler integrals}

\begin{abstract}
We consider the $t$-hook functions on partitions
$f_{a,t}: \mathcal{P}\rightarrow \C$ defined by
$$
f_{a,t}(\lambda):=t^{a-1} \sum_{h\in \mathcal{H}_t(\lambda)}\frac{1}{h^a},
$$
where $\mathcal{H}_t(\lambda)$ is the multiset of partition hook numbers that are multiples of $t$.
The Bloch-Okounkov $q$-brackets $\langle f_{a,t}\rangle_q$  include Eichler integrals of the classical Eisenstein series.
For even  $a\geq 2$, we show that these $q$-brackets are natural pieces of
 weight $2-a$ sesquiharmonic and harmonic Maass forms, while for odd $a\leq -1,$
we show that they are holomorphic quantum modular forms.
We use these results to obtain new formulas of Chowla-Selberg type,
and asymptotic expansions involving values of the Riemann zeta-function and Bernoulli numbers.  We make use of work of Berndt, Han and Ji, and Zagier. 
\end{abstract}

\maketitle
\section{Introduction and statement of results}

A {\it partition} of a non-negative integer $n$ is any nonincreasing sequence of positive integers, say 
 $\lambda=(\lambda_1, \lambda_2, \dots, \lambda_m),$  that satisfies $|\lambda|=\lambda_1+\dots+\lambda_m= n$. 
 Each partition has a {\it Ferrers-Young} diagram
$$
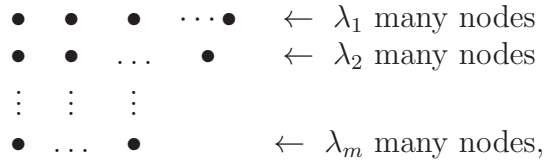
\begin{matrix}
\bullet & \bullet & \bullet & \dots \bullet &\ \leftarrow  \ {\text {\rm $\lambda_1$ many nodes}}\\
\bullet & \bullet &\dots &\bullet & \ \leftarrow \ {\text {\rm $\lambda_2$ many nodes}}\\
\vdots & \vdots & \vdots & \ & \ &  \\
\bullet & \dots & \bullet & \ & \   \leftarrow \ {\text {\rm $\lambda_m$ many nodes}},
\end{matrix}
$$
and each node has a {\it hook number}. The node in row $\ell$ and column $j$ has hook number
$h(\ell,j):=\lambda_\ell-\ell+\lambda'_j-j+1,$ where $\lambda'_j$ is the number of nodes in column $j$.
These numbers play significant roles in combinatorics, number theory, and representation theory. 
The recent Bloch-Okounkov theory of $q$-brackets \cite{BlochOkounkov} is a significant addition to these fields,
and some of the most striking examples involve hook numbers. 

For functions $f: \mathcal{P}\mapsto \C$ on the integer partitions, these $q$-brackets are the power series
\begin{equation*}
\langle f\rangle_q:=\frac{\sum_{\lambda\in \mathcal{P}} f(\lambda)q^{|\lambda|}}{\sum_{\lambda\in \mathcal{P}}q^{|\lambda|}} \in \C[[q]],
\end{equation*}
 which represent a ``weighted average'' of $f.$
 Schneider \cite{Schneider} has developed a ``multiplicative theory of partitions'' based on $q$-brackets,
 which includes partition analogues of classical number theoretic facts such as M\"obius inversion, special values of zeta-functions, etc.
 In this note we add to the expanding role of $q$-brackets and hook numbers that bridges combinatorics and the theory of modular forms.
 
Bloch and Okounkov defined $q$-brackets with the goal of generating spaces of modular forms.
Interpreted as Fourier expansions in $q:=e^{2\pi i z}$, where $z\in \mathbb{H},$ the upper-half of the complex plane, they proved that the ring of quasimodular forms is generated by the $q$-brackets of special  functions $f$ associated to
shifted symmetric polynomials \cite{BlochOkounkov}. This work has been expanded and refined by Zagier \cite{Zagier}, and Griffin, Jameson, and Trebat-Leder subsequently analyzed the $p$-adic aspects of these constructions \cite{GJTL}. 

Nekrasov and Okounkov later obtained striking identities for the modular forms that are powers of
Dedekind's eta-function $\eta(z):=q^{\frac{1}{24}}\prod_{n=1}^{\infty}(1-q^n)$ \cite{NekrasovOkounkov}. 
For $\alpha\in \C$, define the function 
$$
D_{\alpha}(\lambda):=\prod_{h\in \mathcal{H(\lambda)}} \left(1-\frac{\alpha}{h^2}\right),
$$
where $\mathcal{H}(\lambda)$ denotes the multiset of hook numbers of the partition $\lambda$.
A simple reformulation of (6.12) of \cite{NekrasovOkounkov}, using Euler's partition generating function 
\begin{equation*}%\label{Euler}
\sum_{n=0}^{\infty}p(n)q^n=\sum_{\lambda \in \mathcal{P}}q^{|\lambda|}=\prod_{n=1}^{\infty}\frac{1}{1-q^n},
\end{equation*}
 asserts that
$q^{\frac{\alpha}{24}}\cdot \langle D_{\alpha} \rangle_q= \eta(z)^{\alpha}.$
For integers $\alpha,$ these are
weight $\frac\alpha2$ modular forms (see Chapter 1 of \cite{CBMS}).
 
It is natural to ask whether further modular objects arise from $q$-brackets.  In this note we obtain a comprehensive
framework of  nearly modular $q$-brackets, which includes an earlier example by the second author \cite{K}, that correspond to Eisenstein series. Namely, we make use of work of Han and Ji \cite{Han, Han2} to define
a natural infinite family of weighted $t$-hook functions whose $q$-brackets naturally give rise to  harmonic Maass forms, sesquiharmonic Maass forms, and holomorphic quantum modular forms.  

To motivate our results, we first recall the beautiful realization  of the Eisenstein series as $q$-brackets. For a partition $\lambda = (\lambda_{1}, \dots \lambda_{m})$ and $k\in\mathbb N$, let $S_{2k}(\lambda) := \sum_{j =1}^{m} \lambda_{j}^{2k-1}$.  Note that $S_{2}(\lambda) = |\lambda|$ is the ``size" function.  In \cite{Zagier} it was shown that
\begin{equation*}%\label{simple}
\langle S_{2k} \rangle_q=\frac{\sum_{\lambda\in \mathcal{P}} S_{2k}(\lambda) q^{|\lambda|}}{\sum_{\lambda\in \mathcal{P}}q^{|\lambda|}}=
\sum_{n=1}^{\infty}\sigma_{2k-1}(n)q^n=\frac{B_{2k}(1-E_{2k}(z))}{4k},
\end{equation*}
where $\sigma_{\ell}(n):=\sum_{1\leq d\mid n}d^\ell$, $B_{n}$ is the $n$-th Bernoulli number, and 
$$E_{2k}(z):=1-\frac{4k}{B_{2k}}\sum_{n=1}^{\infty}\sigma_{2k-1}(n)q^n
$$
 is the weight $2k$ Eisenstein series. Recall that $E_{2k}(z)$ is a  modular form  when $2k\geq 4.$
 
We realize the {\it Eichler integrals} of these (and other) forms in terms of $q$-brackets of functions involving partition hook lengths.   For $a \in \C,$ we define
\begin{equation*}
\mathcal{E}_{2-a}(z) := \sum_{n =1}^{\infty} \frac{q^n}{n^{a-1}(1-q^n)} = \sum_{n =1}^{\infty} \sigma_{1-a}(n) q^n.
\end{equation*}
If $2k \geq 4$ is even, then $\mathcal{E}_{2-2k}(z)$ is the usual Eichler integral of $E_{2k}(z).$  Although $E_2(z)$ is not a modular form, it is well-known that
$$
E_2^*(z):=E_2(z)-\frac{3}{\pi\cdot \im(z)}
$$
is a non-holomorphic weight two modular form (for example, see Chapter 6 of \cite{BFOR}). Therefore, it is natural to consider its Eichler integral
\begin{equation*}%\label{E}
\mathcal{E}_{0}(z):=\sum_{n=1}^{\infty}\sigma_{-1}(n)q^n.
\end{equation*}
These Eichler integrals enjoy certain modularity properties that were determined by Berndt in the 1970s \cite{Berndt}. Moreover, 
Bettin and Conrey \cite{BC} considered the modularity in the general case where $k\in \C.$

We now turn to the goals of this note. The first goals are to realize these Eichler integrals as $q$-brackets, which we then use to obtain various types of modular forms.
This work extends an example by the second author corresponding to the case of $\mathcal{E}_{0}(z).$ 
To make this precise, we make use of
 $t$-hooks, the hook numbers which are
multiples of $t$. To this end, for each $a \in \C$ and $t\in \N,$ we define $f_{a,t}: \mathcal{P}\mapsto \C$ by
\begin{equation*}
f_{a,t}(\lambda):=t^{a-1} \sum_{h\in \mathcal{H}_t(\lambda)}\frac{1}{h^a},
\end{equation*}
where $\mathcal{H}_t(\lambda)$ is the multiset of hook numbers which are multiples of $t$. 

\begin{example} We consider the partition $\lambda=4+3+1,$
which has Ferrers-Young diagram
$$
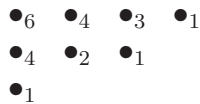
\begin{matrix} \bullet_6 & \bullet_4 & \bullet_3 & \bullet_1 \\
           \bullet_4 & \bullet_2 & \bullet_1\\
           \bullet_1 
           \end{matrix}
$$         
(the subscripts denote the the hook numbers). We find that
$\mathcal{H}(\lambda)=\{1,1,1,2,3,4,4,6\},$ $\mathcal{H}_2(\lambda)=\{2,4,4,6\},$ 
and $\mathcal{H}_3(\lambda)=\{3,6\}.$ Therefore, we find that
\begin{displaymath}
\begin{split}
f_{3,1}(\lambda)&= 1+1+1+\frac{1}{8}+\frac{1}{27}+\frac{1}{64}+\frac{1}{64}+\frac{1}{216}=\frac{307}{96},\\
f_{3,2}(\lambda)&=2^2\left(\frac{1}{8}+\frac{1}{64}+\frac{1}{64}+\frac{1}{216}\right)=\frac{139}{216},\\
f_{3,3}(\lambda)&=3^2\left(\frac{1}{27}+\frac{1}{216}\right)=\frac{3}{8}.
\end{split}
\end{displaymath}
\end{example} 

For convenience, we define the generating function
\begin{equation*}%\label{Ht}
H_{a,t}(z):=\sum_{\lambda\in \mathcal{P}} f_{a,t}(\lambda)q^{|\lambda|}.
\end{equation*}
Work of Han and Ji \cite{Han, Han2} shows that the Eichler integrals
$\mathcal{E}_{2-a}(z)$ are the $q$-brackets of $f_{a,t}.$

\begin{theorem}\label{Theorem1}
If $t\in\N$ and $a \in \C$, then we have
$$
\langle f_{a,t}\rangle_q=\prod_{n=1}^{\infty} (1-q^n)\cdot H_{a,t}(z)=\mathcal{E}_{2-a}(tz).
$$
\end{theorem}

Using the properties of Eichler integrals, we find that many of the $\langle f_{a,t}\rangle_q$ are natural
parts of various types of modular forms, which we now recall.
A weight $k$ \textit{harmonic Maass form} (for example, see \cite{BFOR})  is a real-analytic modular form that is annihilated by the weight $k$ hyperbolic Laplacian $\Delta_{k} := -\xi_{2-k} \circ \xi_{k}$, where $\xi_{k}:= 2iy^{k} \overline{\frac{\partial}{\partial \overline{z}}}$ and that grows at most linear exponentially towards the cusps.  A weight $k$ \textit{sesquiharmonic Maass form} is a real--analytic modular form that is annihilated instead by $\Delta_{k,2} :=- \xi_{k} \circ \xi_{2-k} \circ \xi_{k}$. 
We require the incomplete Gamma function $\Gamma(s,z) := \int_{z}^{\infty} e^{-t} t^{s-1} dt,$ which we normalize to define  $\Gamma^{*}(s,z) := \Gamma(s,z)/\Gamma(s)$. 

\begin{theorem} \label{Theorem2}
If $k\in\mathbb N$, then the following are true.

\begin{enumerate}[leftmargin=*]
	\item[\rm (1)]  If $k=1$, then $\mathbb{E}_{0}(tz)$ is a weight zero sesquiharmonic Maass form on $\Gamma_0(t)$, where
	\begin{equation*}
		\mathbb{E}_{0}(tz) := ty +\frac{6}{\pi} \left( \gamma - \log(2) - \frac{\log(ty)}{2} - \frac{6 \zeta'(2)}{\pi^2} + \langle f_{2, t} \rangle_{q} + \sum_{n =1}^{\infty} \sigma_{-1}(n) \overline{q}^{tn} \right).
	\end{equation*}

	\item[\rm (2)] If $k \geq 2$, then $\mathbb{E}_{2-2k}(tz)$ is a weight $2-2k$ harmonic Maass form on $\Gamma_0(t)$, where
	\begin{multline*}
		\mathbb{E}_{2-2k}(tz)\\ := (ty)^{2k-1}  + \frac{2 \cdot (2k)! }{B_{2k}(4 \pi)^{{2k-1}}} \left(\zeta(2k-1) + \langle f_{2k,t} \rangle_{q} + \sum_{n =1}^{\infty} \sigma_{1-2k}(n) \Gamma^{*}(2k-1, 4 \pi tn y) q^{-tn} \right).
	\end{multline*}
\end{enumerate}
\end{theorem}

\begin{rmk}
Theorem~\ref{Theorem2} {\rm (1)} is a reformulation of an earlier result by the second author in \cite{K}.
\end{rmk}

For completeness, we describe the modularity properties of these $q$-brackets.  For $k\in\mathbb N$, we define
$$
P_{-2k}(z) 
%&:= \frac{1}{2} (2 \pi i)^{2k+1} \sum_{m=0}^{2k+2} \frac{B_{m}}{m!} \frac{B_{2k+2-m}}{(2k+2-m)!} (-z)^{m-1} \\
:= - \frac{1}{2} (2 \pi i)^{2k+1} \sum_{m=0}^{k+1} \frac{B_{2m}}{(2m)!} \frac{B_{2k+2-2m}}{(2k+2-2m)!} \cdot z^{2m-1},
$$
which we use to define
\begin{equation}\label{Mfunction}
M_{-2k, t}(z) := \langle f_{2k+2, t} \rangle_{q} -\frac{1}{2} P_{-2k}(tz) + \frac{1}{2} \zeta(2k+1).
\end{equation}
These functions enjoy negative weight $-2k$ modularity properties under $z\mapsto z+1$ and
$z\mapsto -\frac{1}{t^2z}.$
\begin{theorem}\label{Theorem3}
If $k, t \in \N$, then the following are true for $z \in \mathbb{H}$.

\begin{enumerate}[leftmargin=*]
	\item[\rm (1)] We have that 
	\begin{multline*}
		M_{-2k, t}(z+1) - M_{-2k, t}(z) \\= \frac{1}{4} (2 \pi i)^{2k+1} \sum_{m=0}^{k+1} \sum_{r=1}^{2m-1} \frac{B_{2m}}{(2m)!} \frac{B_{2k+2-2m}}{(2k+2-2m)!} t^{2m-1} \binom{2m-1}{r} z^{2m-1-r}.
	\end{multline*}
	\item [\rm (2)] We have that
	\begin{equation*}
		M_{-2k, t}(z) = (tz)^{2k}M_{-2k,t} \left(-\frac{1}{t^2 z} \right).
	\end{equation*}
\end{enumerate}
\end{theorem}
\begin{rmk}
The case where $k=0$ in (\ref{Mfunction}) was obtained previously in \cite{K}.
\end{rmk}
Theorem \ref{Theorem3} implies certain simple modularity properties 
for the Fourier series
\begin{equation*}
H_{a,1}^{*}(z) := q^{-\frac{1}{24}} H_{a, 1}(z),
\end{equation*}
(i.e., when we choose $t=1$).
To further ease notation, we define
\begin{equation*}
\Psi_{-2k}(z) := -P_{-2k}\left(-\frac{1}{z} \right) - \frac{1}{2}\left(1-z^{-2k}\right) \zeta(2k+1).
\end{equation*}
\begin{corollary}\label{Corollary4}
If $z \in \mathbb{H}$ and $k \in  \N$, then the following are true.

\begin{enumerate}[leftmargin=*]
	\item[\rm (1)] We have that
\begin{equation*}
H_{2k+2, 1}^{*}(z+1) = e^{-\frac{\pi i}{12}} H_{2k+2, 1}^{*}(z).
\end{equation*}

	\item [\rm (2)]  We have that 
\begin{equation*}
H_{2k+2, 1}^{*} \left(-\frac{1}{z} \right)-\frac{1}{z^{2k}\sqrt{-iz}}H_{2k+2, 1}^{*} (z) = \frac{\Psi_{-2k}(z)}{\eta \left(-\frac{1}{z}\right)}.
\end{equation*}
\end{enumerate}
\end{corollary}

Thanks to such transformation laws, we are able to employ
the Chowla-Selberg formula (see \cite{ChowlaSelberg, vW}) to obtain a simple extension of the classical fact that weight $k$ algebraic modular forms evaluated at discriminant $D<0$ points $\tau$ are algebraic multiples of the $k$th power of the canonical period $\Omega_{D}$.  To make this precise, let $\overline{\Q}$ denote the algebraic closure of $\Q.$  Suppose $D<0$ is a fundamental discriminant of the imaginary quadratic field $\Q(\sqrt{D})$ with class number $h(D)$.  Furthermore, define 
\begin{equation*}
h'(D) := \begin{cases} \frac{1}{3} & \text{if} \ D=-3 ,\\ \frac{1}{2} & \text{if} \ D=-4 ,\\ h(D) & \text{if} \ D<-4.
\end{cases}
\end{equation*}
With $\chi_{D}(\cdot) := (\frac{D}{\cdot} )$, we can then define the canonical period by
\begin{equation*}
\Omega_{D} := \frac{1}{\sqrt{2 \pi |D|}} \left( \prod_{j=1}^{|D|} \Gamma \left(\frac{j}{|D|} \right)^{\chi_{D}(j)} \right)^{\frac{1}{2h'(D)}}.
\end{equation*}
We can now state our generalization.
\begin{corollary} \label{Corollary5}
If $k\in\N$ and $\tau \in \Q(\sqrt{D}) \cap \mathbb{H}$, where $D<0$ is a fundamental discriminant, then
\begin{equation*}
H_{2k+2, 1}^{*} \left(-\frac{1}{\tau} \right)-\frac{1}{\tau^{2k}\sqrt{-i\tau}}H_{2k+2, 1}^{*} (\tau) \in \overline{\Q} \cdot \frac{\Psi_{-2k}(\tau)}{\sqrt{\Omega_{D}}}.
\end{equation*}
\end{corollary}
\begin{rmk}
The second author obtained the $k=0$ extension of Corollary~\ref{Corollary5} in \cite{K}.
As noted in \cite{K}, these results can be extended to $t >1$.
\end{rmk}

Theorems 1.2 and 1.3, and Corollaries 1.4 and 1.5 above pertain to the $t$-hook functions $f_{a,t}$, where $a\geq 2$ is even. There is a theory for negative odd $a$.
These cases give rise to quantum modular forms.  Introduced by Zagier in \cite{ZQ}, a weight $k$ \textit{quantum modular form} is vaguely defined as a function $f: \Q \setminus S \to \C$, for some finite set $S$, where 
\begin{equation*}
h_{f,\gamma}(x) := f(x) - (cx + d)^{-k} f \footnotesize{\left(\frac{ax + b}{cx +d} \right)},
\end{equation*} 
with $\gamma = \left(\begin{smallmatrix} a & b \\ c & d \end{smallmatrix}\right) \in \SL_{2}(\Z),$ is ``better behaved" analytically than $f$.  Recently, Zagier \cite{ZV} defined the notion of a  weight $k$ \textit{holomorphic quantum modular form.} These are holomorphic functions $f: \mathbb{H} \to \C,$ where the ``better behaved" condition means that $h_{f, \gamma}(z)$ is holomorphic on a larger domain than $\mathbb{H}$.

By applying mutatis mutandis a method introduced by Zagier in \cite{ZV}, who considered $\mathcal{E}_{3}(z),$ we have the following infinite family of holomorphic quantum modular forms.

\begin{theorem}\label{Theorem6}
If $a\leq -1$ is odd, then the following are true.

\begin{enumerate}[leftmargin=*]
	\item[\rm (1)]  We have that $\langle f_{a,t} \rangle_{q}$ is a holomorphic weight $2-a$ quantum modular form.  
In particular, we have the modular transformations
\begin{displaymath}
\begin{split}
\mathcal{E}_{2-a}(z) - \mathcal{E}_{2-a}(z+1) &= 0, \\
\mathcal{E}_{2-a}(z) - z^{a-2} \mathcal{E}_{2-a} \left(-\frac{1}{z} \right) &=\frac{1}{2 \pi} \int_{\re(s) = 1 - \frac{a}{2}} \frac{\Gamma(s) \zeta(s) \zeta(s+a-1)}{(2 \pi)^{s} \sin \left( \frac{\pi s}{2} \right)}z^{-s} ds \\
&= 2 \sideset{}{^{'}}\sum_{m,n \geq 0}%\!{'} 
\frac{1}{(mz+n)^{2-a}},
\end{split}
\end{displaymath}
where the $'$ denotes that the terms where $m$ or $n$ (but not both) equal zero are weighted by $\frac{1}{2}$.

	\item[\rm (2)] As $t \to 0^{+},$ we have the asymptotic expansion
\begin{equation*}
\mathcal{E}_{2-a} \left( \frac{it}{2 \pi} \right) \sim \frac{\Gamma(2-a) \zeta(2-a)}{t^{2-a}} + \frac{\zeta(a)}{t} + \sum_{n = 0}^\infty \frac{B_{n+1}}{n+1} \frac{B_{n+2-a}}{n+2-a} \frac{(-t)^n}{n!}.
\end{equation*}
\end{enumerate}
\end{theorem}
\begin{rmk}
Part (2) is already known by work of Zagier. To be more precise he showed (2) in \cite{Za}, using the Euler--Maclaurin summation formula.  However, for the readers convenience, we include a different (and instructive) proof here, which is also an adaptation of an argument of Zagier.
\end{rmk}
\begin{rmk}
For $a=1,$ similar results hold. Namely, we have
\begin{equation*}
\mathcal{E}_{1} \left( \frac{it}{2 \pi} \right) \sim  \frac{2 \gamma}{t} + \sum_{n = 0}^\infty \frac{B_{n+1}^2}{(n+1)^2} \frac{(-t)^n}{n!},
\end{equation*}
where $\gamma$ is the Euler-Mascheroni constant.  One can also use similar methods to find an asymptotic expansion for $\mathcal{E}_{2-a} ( \alpha + \frac{it}{2 \pi} )$ with $\alpha \in \Q$.
\end{rmk}

\begin{rmk}
This new kind of quantum modularity was also noted by Bettin and Conrey in \cite{BC} where they computed $\mathcal{E}_{k}(z) - z^{-k} \mathcal{E}_{k} (-\frac{1}{z} )$ for any $k \in \C$.  Folsom has recently extended their work in \cite{AF}, where it was shown that a new family of ``twisted Eisenstein series'' are holomorphic quantum modular forms, which were then used to show that certain cotangent-zeta sums are quantum modular forms in the original sense.  Zagier's presentation in \cite{ZV} shows that for any $\gamma = \left(\begin{smallmatrix} a & b \\ c & d \end{smallmatrix}\right) \in \SL_{2}(\Z)$ that $h_{\mathcal{E}_{k}, \gamma}(z)$ extends to a holomorphic function on the cut plane
\begin{equation*}
\C_{\gamma} := \begin{cases} \C \setminus \left(-\infty, -\frac{d}{c}\right) & c>0, \\
\C \setminus\left(-\frac{d}{c}, \infty\right) & c<0. \end{cases} \end{equation*}
\end{rmk}

In Section 2 we recall essential $q$-identity preliminaries, and in Section 3 we prove the theorems.
In the last section we offer numerical examples of these results.

\section*{Acknowledgements} \noindent The authors thank Amanda Folsom and Wei-Lun Tsai for their comments on preliminary versions of this paper. Moreover we thank the referees for helpful comments.

\section{Nuts and Bolts}\label{NutsAndBolts}
We recall work of Han and Ji that is integral to the proof of Theorem \ref{Theorem1}.
\begin{theorem}[Theorem 7.5 of \cite{Han2}] \label{Han2}
For an $k \in \C$ and positive integer $t$ we have
\begin{equation*}
t^{k-1}\sum_{\lambda \in \mathcal{P}} q^{|\lambda|} x^{|\mathcal{H}_{t}(\lambda)|}  \sum_{h \in \mathcal{H}_{t}(\lambda)} \frac{1}{h^{k}} = \prod_{n = 1}^\infty \frac{(1-q^{tn})^{t}}{(1-x^{n}q^{tn})^{t} (1-q^n)} \sum_{n = 1}^\infty \frac{x^{n} q^{tn}}{n^{k-1}(1-x^{n} q^{tn})}.
\end{equation*}
\end{theorem}
We now recall a theorem of Berndt which is used to prove Theorem \ref{Theorem3}.
\begin{theorem}[Theorem 2.2 of \cite{Berndt}] \label{B}
For $z \in \mathbb{H}$ and $k\in\N$ we have
\begin{equation*}
\mathcal{E}_{-2k}(z) - z^{2k} \mathcal{E}_{-2k} \left( -\frac{1}{z} \right) = -\frac{1}{2} \left(1- z^{2k} \right) \zeta(2k+1) - P_{-2k}(z).
\end{equation*}
\end{theorem}

\begin{comment}
The following Theorem of Bettin and Conrey establishes the quantum modular behavior of the $\mathcal{E}_{k}$.  We will also give two other direct proofs give by Zagier in \cite{ZV} in the next section which are greatly simplified for the odd integral $k$ cases.
\begin{theorem}[Theorem 1 of \cite{BC}] \label{BC}
Let $Im(z) >0$ and $k \in \C$ and define
\begin{equation*}
\psi_{k}(z) := \mathcal{E}_{k}(z) - z^{-k} \mathcal{E}_{k} \left(-\frac{1}{z} \right).
\end{equation*}
Then $\psi_{k}(z)$ extends to an analytic function on $\C' := \C \setminus \R_{\leq 0}$ via the representation
\begin{equation*}
\phi_{k}(z) = \frac{i \zeta(2-k)}{2 \pi z} - \frac{\zeta(1-k)}{2} + \frac{e^{\frac{\pi i k}{2}} \zeta(k) \Gamma(k)}{(2 \pi z)^k} + \frac{i}{2} g_{k-1}(z),
\end{equation*}
where
\begin{align*}
g_{k-1}(z) &= -2 \sum_{1 \leq n \leq M} (-1)^n \frac{B_{2n}}{(2n)!} \zeta(2 -2n -k) (2 \pi z)^{2n-1} \\
&+ \int_{(-1/2 -2M)} \zeta(s) \zeta(s-k+1) \Gamma(s) \frac{\cos \left(\frac{\pi(k-1)}{2} \right)}{\sin \left( \frac{\pi(s-k+1)}{2} \right)} (2 \pi z)^{-s} ds,
\end{align*}
with $M \geq -\frac{1}{2} \rm{min}(0, Re(k-1))$.
\end{theorem}
\end{comment}

\section{Proofs of the Theorems}

\begin{proof}[Proof of Theorem \ref{Theorem1}]
We set $x=1$ in Theorem \ref{Han2} to obtain
\begin{equation*}
t^{k-1} \sum_{\lambda \in \mathcal{P}} q^{|\lambda|}  \sum_{h \in \mathcal{H}_{t}(\lambda)} \frac{1}{h^k} = \prod_{n = 1}^\infty \frac{1}{1-q^n} \sum_{n = 1}^\infty \frac{q^{tn}}{n^{k-1} (1-q^{tn})}.
\end{equation*}
The statement is then immediate from the definition of the $q$-bracket.
\end{proof}

\begin{proof}[Proof of Theorem \ref{Theorem2}]
	(1)
Note that $E_{2}^{*}(z)$ is a limit of an Eisenstein series
	\begin{align*}%\label{E2hat}
	E_{2}^{*}(z)&=\lim_{s \to 0} \sum_{M \in \Gamma_{\infty} \setminus \SL_{2}(\Z)}y^s\big|_{2}M(z).
	\end{align*}
	For $\operatorname{Re}(s)>1,$ we recall the Eisenstein series
	\begin{equation*}
	E(z,s):=\frac 12\sum_{\gcd(c,d)=1}\frac{y^s}{|cz+d|^{2s}}=\sum_{M\in\Gamma_\infty\setminus\SL_2(\Z)} y^s\big|_0M(z).
	\end{equation*}
	Then $E(z,s)$ has a meromorphic continuation to the whole $s$-plane again denoted by $E(z,s)$ with a simple pole with residue $\frac 3\pi$ at $s=1$. We now define
	\begin{equation*}
	\mathbb{E}(z,s):=E(z,s)-\frac{3}{\pi(s-1)} \ \ \ 
	{\text {\rm and}}\ \ \  \widetilde{\mathbb{E}}_{0}(z):=\lim_{s\to 1} \mathbb{E}(z,s).
	\end{equation*}
	Below we show that $\widetilde{\mathbb{E}}_0=\mathbb{E}_0$.
	A direct calculation shows that
	\begin{align*}
	D\left(\widetilde{\mathbb{E}}_{0}\right)=-\frac{1}{4\pi} E_{2}^{*}.
	\end{align*}
	
	\noindent Clearly, $\widetilde{\mathbb{E}}_{0}(z)$ has weight zero, since $E(z,s)$ does. Moreover, recalling that
	\begin{align*}
	\Delta_{0}(E(z,s))=s(1-s)E(z,s),
	\end{align*}
	we obtain that
	\begin{align*}
	\Delta_{0}\left(\widetilde{\mathbb{E}}_{0}\right)(z)=-\frac{3}{\pi}.
	\end{align*}
Therefore, the function is sesquiharmonic.\\
	\indent We now compute its Fourier expansion. We have
	\begin{equation*}
	E(z,s)=y^s+\frac{\zeta^{*}(2s-1)}{\zeta^{*}(2s)}y^{1-s}
	+\frac{4\sqrt{y}}{\zeta^*(2s)}\sum_{m=1}^{\infty}m^{s-\frac{1}{2}}\sigma_{1-2s}(m)K_{s-\frac 12}(2\pi my)\cos(2\pi m x),
	\end{equation*}
	where
	\begin{equation*}
	\zeta^{*}(s):=\pi^{-\frac s2}\Gamma\left(\frac s2\right) \zeta(s).
	\end{equation*}
	We need to take $s\to 1$ and subtract $\frac{3}{\pi(s-1)}$ from the constant term.
	First note that in the sum on $m$, we can just plug in $s=1$ to evaluate
	\begin{equation*}
		\frac 6\pi\sum_{m=1}^{\infty}\sigma_{-1}(m)\left(q^m+\overline{q}^{m}\right),
	\end{equation*}
	using
$
	K_{\frac 12}(x)=\sqrt{\frac{\pi}{2x}} e^{-x}.
$ Finally, we directly compute
	\begin{align*}
	\lim_{s\to 1}\left(\frac{\zeta^{*}(2s-1)}{\zeta^{*}(2s)}y^{1-s}-\frac{3}{\pi(s-1)}\right)
	=
	\frac{6\gamma}{\pi}-\frac{6\log(2)}{\pi}-\frac{3\log(y)}{\pi}-\frac{36}{\pi^3}\zeta'(2).
	\end{align*}
	Therefore, we obtain
	\begin{equation*}
	\widetilde{\mathbb{E}}_{0}(z)=y+\frac{6\gamma}{\pi}-\frac{6\log(2)}{\pi}-\frac{3\log(y)}{\pi}-\frac{36}{\pi^3}\zeta'(2)
	+\frac 6\pi\sum_{m=1}^{\infty}\sigma_{-1}(m)\left(q^m+\overline{q}^{m}\right)=\mathbb{E}_0(z).	
\end{equation*}
The linear exponential growth in $i\infty$ follows directly from the Fourier expansion; the other cusps can be treated in a similar manner.
	
\noindent (2) The proof is well-known (for example, see Corollary 6.16 of \cite{BFOR}).
\end{proof}

\begin{proof}[Proof of Theorem \ref{Theorem3}]
By Theorem \ref{Theorem1}, we have 
\begin{equation*}
M_{-2k,t}(z) = \mathcal{E}_{-2k}(tz) - \frac{1}{2} P_{-2k}(tz) + \frac{1}{2} \zeta(2k+1).
\end{equation*}
Letting $z \mapsto tz$ in Theorem \ref{B}, we obtain
\begin{equation*}
\mathcal{E}_{-2k}(tz) - (tz)^{2k} \mathcal{E}_{-2k} \left(-\frac{1}{tz} \right) = \frac{1}{2} \left((tz)^{2k} -1\right) \zeta(2k+1) + P_{-2k}(tz).
\end{equation*}
By a direct calculation, we find that $(tz)^{2k} P_{-2k} (-\frac{1}{tz} ) =-P_{-2k}(tz),$ and thus
\begin{align*}
M_{-2k,t}(z) - &(tz)^{2k} M_{-2k,t} \left(-\frac{1}{t^2 z} \right) \\
&= \mathcal{E}_{-2k}(tz) - (tz)^{2k} \mathcal{E}_{-2k}\left( -\frac{1}{tz} \right) -\frac{1}{2} \left( P_{-2k}(tz) - (tz)^{2k} P_{-2k} \left(-\frac{1}{tz} \right) \right) 
\\&\hspace{130pt}+ \frac{1}{2}\left(1 - (tz)^{2k}\right) \zeta(2k+1) 
\\
&= \frac{1}{2}((tz)^{2k} -1) \zeta(2k+1) + P_{-2k}(tz) - P_{-2k}(tz) + \frac{1}{2}\left(1 - (tz)^{2k}\right) \zeta(2k+1) =0.
\end{align*}
This gives part~(2).

Because $\mathcal{E}_{-2k}(tz)$ is invariant under $z \mapsto z+1$ part (1) follows from computing $P_{-2k}(tz +t) - P_{-2k}(tz)$.
\end{proof}

\begin{proof}[Proof of Corollary \ref{Corollary4}]
We have that $H_{2k+2,1}^{*}(z) = \frac{\mathcal{E}_{-2k}(z)}{\eta(z)},$ and so the corollary follows from the fact that \begin{equation*}
\mathcal{E}_{-2k} \left(-\frac{1}{z} \right) - z^{-2k} \mathcal{E}_{-2k}(z) = \Psi_{-2k}(z),
\end{equation*}
and the transformation properties of the Dedekind eta-function.
\end{proof}
\begin{proof}[Proof of Corollary \ref{Corollary5}]
By the classical Chowla-Selberg Theorem (see \cite{ChowlaSelberg, vW}), we have that
\begin{equation*}
\eta \left(-\frac{1}{\tau} \right) \in \overline{\Q} \cdot \sqrt{\Omega_{D}}.
\end{equation*}
Corollary \ref{Corollary5} is now a consequence of Corollary \ref{Corollary4}.
\end{proof}

\begin{proof}[Proof of Theorem \ref{Theorem6}]
For odd $k \geq 1,$ we define
\begin{equation} \label{Gk}
G_{k}(z) := -\frac{B_{k}}{2k} + \sum_{n =1}^\infty \sigma_{k-1}(n) q^n.
\end{equation}
This is a slight abuse of notation as $G_{k}(z) = \mathcal{E}_{k}(z)$ in these cases (except for $k=1$) because $B_{k}=0$ for $k \geq 3$ and odd.  We let 
\begin{equation*}
\widetilde{G}_{k}(s) := \int_{0}^{\infty} \left( G_{k}(iy) + \frac{B_{k}}{2k} \right) y^{s-1} dy = \frac{\Gamma(s)}{(2 \pi)^s} \zeta(s) \zeta(s-k+1).
\end{equation*} \\
Before we prove the quantum modularity, we address the claimed asymptotic expansions. The idea, which is well-known (for example, see \cite{LZ2} or the proof of Theorem 21.4 in \cite{BFOR}), is to relate the product of zeta functions on the right to the Mellin integral representation involving $G_{k}(z)$.  The following argument follows almost mutatis mutandis as in pages 99-100 of \cite{LZ2}.  To compute these expansions, we write $G_{k} ( \frac{it}{2 \pi} )$ as a contour integral 
\begin{equation*}
\frac{1}{2 \pi i} \int_{C} \Gamma(s) \zeta(s) \zeta(s-k+1) t^{-s} ds,
\end{equation*} 
where the contour encircles the negative imaginary axis.  Moving the contour across the simple poles of $\Gamma(s) \zeta(s) \zeta(s-k+1)$ gives the desired expansions due to the poles at $s=1$ and $k$ for the zeta function factor, and the poles at each nonpositive integer for the gamma function.  In this way we obtain the formula
\begin{equation*}
\mathcal{E}_{k} \left( \frac{it}{2 \pi} \right) \sim \frac{\Gamma(k) \zeta(k)}{t^{k}} + \frac{\zeta(2-k)}{t} + \sum_{n = 0}^\infty \frac{B_{n+1}}{n+1} \frac{B_{n+k}}{n+k} \frac{(-t)^n}{n!}.
\end{equation*} 
Namely, the first two summands correspond to the residues at $s=k$ and $s=1$, while the remaining sums involving Bernoulli numbers correspond the poles arising to the gamma function and the special values of the zeta function at negative integers.

In order to prove Theorem \ref{Theorem6} $(1)$, we adapt the proof of Proposition 10 in \cite{LZ} which pertained to a similar function.  We first note that we can extend $G_{k}(z)$ to $\C \setminus \R$ by defining
\begin{equation*}
G_{k}(z) := \begin{cases} -\frac{B_{k}}{2k} + \sum_{n=1}^{\infty} \sigma_{k-1}(n) q^{n} & \text{if } {\rm{Im}}(z) >0, \\ \frac{B_{k}}{2k} - \sum_{n = 1}^{\infty} \sigma_{k-1}(n) q^{-n} & \text{if } {\rm{Im}}(z) <0.
\end{cases}
\end{equation*}
Then we define the period function
\begin{equation*}
\psi_{k} (z) := G_{k}(z) - z^{-k} G_{k} \left(-\frac{1}{z} \right).
\end{equation*}
The Mellin transforms of $G_{k}$ and $\psi_{k}$ restricted to the positive or negative imaginary axis are given by
\begin{align*}
\widetilde{G}_{k, \pm}(s) &:= \int_{0}^{\infty} \left( G_{k}(\pm iy) \pm \frac{B_{k}}{2k} \right) y^{s-1} dy = \pm \frac{\Gamma(s)}{(2 \pi)^{s}} \zeta(s) \zeta(s-k+1),\\
\widetilde{\psi}_{k, \pm}(s) &:= \int_{0}^{\infty} \psi_{k}(\pm i y) y^{s-1} dy = \int_{0}^{\infty} \left(G_{k}(\pm iy) \pm \frac{i^k}{y^k} G_{k} \left( \pm \frac{i}{y} \right) \right)y^{s-1} dy \\
&= \widetilde{G}_{k, \pm}(s) \pm i^{k} \widetilde{G}_{k, \pm}(k-s) = \left(1 \mp \frac{e^{\frac{\pi i s}{2}} + e^{-\frac{\pi i s}{2}}}{e^{\frac{\pi i s}{2}} - e^{-\frac{\pi is}{2}}} \right) \widetilde{G}_{k, \pm}(s)  \\
&= \frac{i e^{\mp \pi i \frac{s}{2}}\Gamma(s) \zeta(s) \zeta(s-k+1)}{(2 \pi)^{s}\sin \left(\frac{\pi s}{2} \right)}.
\end{align*}
Using the Mellin inversion formula, we obtain
\begin{equation*}
\psi_{k}( \pm iy) = \frac{1}{2 \pi } \int_{{\rm{Re}}(s) = \frac{k}{2}} \frac{\Gamma(s) \zeta(s) \zeta(s-k+1)}{(2 \pi)^{s}\sin \left(\frac{\pi s}{2} \right)} (\pm iy)^{-s} ds
\end{equation*}
 for $y>0$.  As in \cite{LZ}, by analytic continuation from $i \R \setminus \{0 \}$ to $\C \setminus \R,$ we have
\begin{equation*}
\psi_{k}(z) = \frac{1}{2 \pi} \int_{{\rm{Re}}(s) = \frac{k}{2}} \frac{\Gamma(s) \zeta(s) \zeta(s-k+1)}{(2 \pi)^{s}\sin \left(\frac{\pi s}{2} \right)}  z^{-s} ds
\end{equation*}
for $z \in \C \setminus \R$.
We note that the fraction is bounded by a power of $s$ times $e^{-\pi s},$ as $|s| \to \infty$ on vertical strips. Therefore, by writing $z^{-s} = e^{-s \Log(z)},$ we see that the integral converges for $|\Arg(z)| < \pi.$ Therefore, it is holomorphic on the cut plane $\C' := \C \setminus \R^-$ .

Zagier noted in \cite{ZV} that for odd $k \geq 3,$ that we can also write
\begin{equation*}
G_{k}(z) = \sum_{(m,n) \in \Lambda^{+}} \frac{1}{(mz+n)^k},
\end{equation*}
where
\begin{equation*}
\Lambda^{+} := \left\{ (m,n) \in \Z^2 : m>0 \ \text{or} \ m=0 \ \text{and} \ n>0 \right\}.
\end{equation*}
Determining $G_{k} - G_{k} | \gamma$ then boils down to a determination of the action of $\gamma$ on $\Lambda^{+}$.  The case of $\gamma = S=
\left(\begin{smallmatrix} 0&-1\\ 1&0\end{smallmatrix}\right)$ is straightforward as it is merely a 90 degree counterclockwise rotation of $\Lambda^{+}.$ This corresponds to the second and third quadrants instead of the first and second.  The difference then cancels the terms in the second quadrant while summing over the third quadrant is the same as the first quadrant just with a sign change.  Combining these observations then gives
\begin{equation*}
G_{k}(z) - z^{-k} G_{k} \left(-\frac{1}{z} \right) = 2 \sideset{}{^{'}}\sum_{m,n \geq 0}\frac{1}{(mz+n)^k} \in \Hol(\C').
\end{equation*}
The holomorphicity of this expression follows from the fact that poles cannot be  introduced by choosing $z \in \C'$, but if $z \in \R^-$ then $mz+n$ can be arbitrarily small causing the sum to diverge.
\end{proof}

\section{Examples}

\begin{example}
We now illustrate Corollary \ref{Corollary5} using $z = \tau = i$ and $z=\tau=2i$.  Ramanujan proved (see p. 326 of \cite{Berndt1998}) that
\begin{equation*}
\eta\left(\frac i2\right) = 2^{\frac{1}{8}} \cdot \sqrt{\Omega_{-4}} \approx 0.8377.
\end{equation*}
By direct calculation we find
\begin{align*}
\Psi_{-2}(2i) &= \frac{37 \pi^3}{1440} - \frac{5 \zeta(3)}{8} \approx 0.04540, &&
\Psi_{-4}(2i) =  \frac{\pi^5}{576} - \frac{15 \zeta(5)}{32} \approx 0.04522.
\end{align*}
We therefore have
\begin{align*}
\frac{\Psi_{-2}(2i)}{\eta\left(\frac i2\right)} &\approx 0.05420, &&
\frac{\Psi_{-4}(2i)}{\eta\left(\frac i2\right)} \approx 0.05398.
\end{align*}
By direct calculation, we find that
\begin{align*}
H_{4,1}^{*}(2i) &\approx 5.887 \cdot 10^{-6}, \qquad && H_{4,1}^{*}\left(\frac i2\right) \approx 0.05420, \\
H_{6,1}^{*}(2i) & \approx 5.887 \cdot 10^{-6}, \qquad && H_{6,1}^{*} \left(\frac i2\right)\approx 0.05398.
\end{align*}
Combining these calculations gives us
\begin{align*}
H_{4,1}^{*}\left(\frac i2\right) + \frac{1}{2^{\frac{5}{2}}} H_{4,1}^{*}(2i) = \frac{\Psi_{-2}(2i)}{\eta\left(\frac i2\right)} = \frac{1}{2^{\frac{1}{8}}} \cdot \frac{\Psi_{-2}(2i)}{\sqrt{\Omega_{-4}}},
\\
H_{6,1}^{*}\left(\frac i2\right) - \frac{1}{2^{\frac{9}{2}}} H_{6,1}^{*}(2i) = \frac{\Psi_{-4}(2i)}{\eta\left(\frac i2\right)} = \frac{1}{2^{\frac{1}{8}}} \cdot \frac{\Psi_{-4}(2i)}{\sqrt{\Omega_{-4}}}
\end{align*}
so the algebraic factor is $2^{-\frac{1}{8}}$ in both cases.

\end{example}

\begin{example}
We now illustrate Theorems \ref{Theorem1} and \ref{Theorem6}. 
We recall that Theorem~\ref{Theorem1} implies that
$$
\langle f_{a,1}\rangle_q=\prod_{n=1}^{\infty} (1-q^n)\cdot H_{a,1}(z)=\mathcal{E}_{2-a}(z).
$$
Therefore, if $k=2-a,$ where $a\leq -1$ is odd, then Theorem~\ref{Theorem6} (2) asserts that the functions
\begin{equation*}
\widehat{G}_{k}(t) := \sum_{n =1}^\infty \sigma_{k-1}(n) e^{-nt}=\mathcal{E}_{2-a}\left(\frac{it}{2\pi}\right),
\end{equation*}
and
\begin{equation*}
\widetilde{G}_{k}(t) := \frac{\Gamma(k) \zeta(k)}{t^k} + \frac{\zeta(2-k)}{t} + \sum_{n =0}^\infty \frac{B_{n+1}}{n+1} \frac{B_{n+k}}{n+k} \frac{(-t)^{n}}{n!}.
\end{equation*}
have the same asymptotic behavior as $t \to 0^+.$
The table below illustrates this when $a=-1.$

\smallskip
\begin{center}
\def\arraystretch{1.5}
  \begin{tabular}{ | c | c | c | c | }
    \hline
    $t$ & $\widehat{G}_{3}(t)$ & $\widetilde{G}_{3}(t)$ & $\widehat{G}_{3}(t)/\widetilde{G}_{3}(t)$ \\ \hline
    $2$ & $\approx 0.2602861623$ & $\approx 0.2602864321$ & $\approx 0.9999989634$  \\ 
    $1.5$ & $\approx 0.6578359053$ & $\approx 0.6578359052$ & $\approx 0.9999999998$ \\ 
    $1$ & $\approx 2.3214805734$ & $\approx 2.3214805734$ & $\approx 1.0000000000$ \\ 
    $0.5$ & $\approx 19.0665916994$ & $\approx 19.0665916994$ & $\approx 1.0000000000$ \\ 
    $0.1$ & $\approx 2403.2805424358$ & $\approx 2403.2805424358$ & $\approx 1.0000000000$ \\ 
   
    \hline
  \end{tabular}
\end{center}
\end{example}

\end{document}